\documentclass[draft]{article}

\usepackage{amsthm,amsmath,amssymb}
\usepackage[mathscr]{eucal}

\theoremstyle{plain}
\newtheorem{theorem}{Theorem}[section]
\newtheorem{lemma}[theorem]{Lemma}
\newtheorem{corollary}[theorem]{Corollary}
\newtheorem{proposition}[theorem]{Proposition}
\theoremstyle{definition}
\newtheorem{definition}[theorem]{Definition}
\theoremstyle{remark}
\newtheorem{remark}[theorem]{Remark}
\newtheorem{example}[theorem]{Example}

\newcommand{\Lie}[1]{\operatorname{\textsl{#1}}}

\newcommand{\Ad}{\textup{Ad}}
\newcommand{\ad}{\textup{ad}}
\newcommand{\Aut}{\mathop{\mbox{\rm Aut}}}

\newcommand{\C}{\mathbb{C}}

\newcommand{\cT}{\mathcal{T}}

\newcommand{\Cyclic}{\mathop{\text{\Large$\mathfrak S$}\vrule width 0pt depth 2pt}}
\newcommand{\cyclic}{\mathop{\text{\large$\mathfrak S$}\vrule width 0pt depth 2pt}}

\newcommand{\D}{\mathrm{d}}
\newcommand{\diag}{\mathrm{diag}}
\newcommand{\E}{\mathrm{e}}

\newcommand{\fe}{\mathfrak{e}}

\newcommand{\fg}{\mathfrak{g}}
\newcommand{\fh}{\mathfrak{h}}

\newcommand{\fm}{\mathfrak{m}}
\newcommand{\fp}{\mathfrak{p}}

\newcommand{\fsl}{\mathfrak{sl}}
\newcommand{\fso}{\mathfrak{so}}
\newcommand{\fsp}{\mathfrak{sp}}
\newcommand{\fsu}{\mathfrak{su}}

\newcommand{\fu}{\mathfrak{u}}

\newcommand{\I}{\mathrm{i}}

\newcommand{\na}{\nabla}
\newcommand{\R}{\mathbb{R}}

\newcommand{\Ric}{\mathrm{Ric}}

\newcommand{\SL}{\Lie{SL}}

\newcommand{\SU}{\Lie{SU}}

\newcommand{\T}{\mathcal{T}}
\newcommand{\tr}{\mathrm{tr}\,}

\newcommand{\va}{\varepsilon}

\newcommand{\wn}{\widetilde{\nabla}}
\newcommand{\wR}{\widetilde{R}}
\newcommand{\wT}{\widetilde{T}}

\newcommand{\comp}{\makebox[7pt]{\raisebox{1.5pt}{\tiny $\circ$}}}

\begin{document}

\title{\bf Cyclic metric Lie groups\thanks{The first and third authors have been  supported by the Ministry of Economy and Competitiveness, Spain, under Project MTM2011--22528. The second author has been  supported by D.G.I. (Spain) and FEDER Projects MTM2010--15444 and MTM2013-46961-P. }}

\author{P.~M.~Gadea, J.~C.~Gonz{\'a}lez-D{\'a}vila and J.~A.~Oubi\~na}

\maketitle

\begin{abstract}
Cyclic metric Lie groups are Lie groups equipped with a left-invariant metric which is in some way far from being biinvariant,
in a sense made explicit in terms of Tricerri and Vanhecke's homogeneous structures.
The semisimple and solvable cases are studied.  We extend to the general case, Kowalski-Tricerri's and Bieszk's classifications of connected and simply-connected unimodular 
cyclic metric Lie groups for dimensions less than or equal to five. 
\end{abstract}

{\footnotesize

\noindent {\it A.M.S. Subjetc Classification 2010:} 53C30, 22E25, 22E46.
}

\section{Introduction}
\label{intro} 
For any metric Lie group $(G,g)$, let $\fg$ be its Lie algebra, let $\na$ denote
the Levi-Civita connection and consider the left-invariant homogeneous structure $S$ on $(G,g)$ defined by $S_XY = \na_XY$, $X,Y \in \fg$.

We say that $(G,g)$ is {\it cyclic} if the $g$-torsion of the (-)-connection of Cartan-Schou\-ten $\wn:=\nabla - S$ is cyclic (see \cite{Puh}). This means that $S\in \T_1\oplus\T_2$ in
Tricerri and Vanhecke's classification (see \cite{TriVan}) or, equivalently, the corresponding inner product $\langle\cdot,\cdot\rangle$ on $\fg$ satisfies
\[
\Cyclic_{XYZ}\langle [X,Y],Z\rangle =0,
\]
for all $X,Y,Z\in \fg$. So nonabelian cyclic metric Lie groups are in  some way far from being biinvariant Lie groups,
for which the similarly defined structure $S$ belongs to the last basic Tricerri and Vanhecke's class $\T_3$, or equivalently, the torsion of $\wn$ is totally skew-symmetric.

The homogeneous Riemannian spaces $G/H$ that generalize  in a natural way  the cyclic metric Lie groups are the {\it cyclic} homogeneous Riemannian manifolds, that we have studied in~\cite{GadGonOub1}. 

Note that Pf\"affle and Stephan \cite[p.\ 267]{PfaSte} (see also Friedrich \cite{Fri}) proved that the Dirac operator on a Riemannian manifold $(M,g)$ does not contain any information on the Cartan-type component of the torsion of a generic Riemannian connection. In the case of homogeneous Riemannian manifolds, this is the $\cT_2$-component of a homogeneous structure on the manifold. Hence, the connected and simply-connected traceless cyclic metric spin Lie groups have the simplest Dirac operator (that is, like that on a Riemannian symmetric spin space) among connected and simply-connected metric spin Lie groups.    

The paper is organized as follows. In Section \ref{secdos} we give some  preliminaries on homogeneous Riemannian structures.

In Section \ref{sectre} we consider cyclic metric Lie groups and exhibit some differences with biinvariant Lie groups.
After proving (Proposition \ref{pA}) that a connected cyclic metric Lie group is flat if and only if it is abelian,
we study (Proposition \ref{behav}) the different properties of sectional curvatures and scalar curvature according to the group being
solvable, unimodular or nonunimodular.

The semisimple case is studied in Section \ref{seccua}. We prove (Lemma \ref{cslg}, Theorem   \ref{theo34}) that every semisimple cyclic metric Lie group (and, more generally, every nonabelian cyclic metric Lie group) is not compact. Then we show (Corollary \ref{noposi}) that no cyclic metric Lie group has strictly positive sectional curvature and (Theorem   \ref{simple}) that the universal covering group of $SL(2,\R)$ is the only connected, simply-connected simple real cyclic metric Lie group.

In Section \ref{seccin} we consider the solvable case. After obtaining (formula \eqref{solv} and Proposition \ref{pco}) some specific properties useful in the classification for low dimensions, we characterize (Lemma \ref{semidi}) the cyclic metric Lie groups  among semidirect products of cyclic metric Lie groups. Furthermore, we show (Proposition \ref{psplit}) that any solvable cyclic metric Lie group can be expressed as an orthogonal semidirect product. After proving (Proposition \ref{nabnilp}) that nonabelian nilpotent Lie groups do not admit any cyclic left-invariant metric, we construct (Examples~\ref{solvable-gen-1} and~\ref{solvable-gen-2}) two multi-parameter families of solvable cyclic metric Lie groups. 

The classification of connected, simply-connected nonabelian Lie groups $G$ admitting a (nontrivial) traceless cyclic homogeneous structure, that is, for unimodular $G$, was given for dimensions three and four by Kowalski and Tricerri \cite[Theorems  2.1, 3.1]{KowTri} and for (Lie algebras of) dimension five by Bieszk \cite{Bie}. We extend their results, adding the corresponding nonunimodular Lie groups, in Theorems \ref{classi}, \ref{four-nonunim} and \ref{theofive}.

\section{Preliminaries}
\label{secdos}

A  homogeneous structure on a Riemannian manifold $(M,g)$  is a tensor field $S$ of type $(1,2)$ satisfying $\widetilde{\nabla} g = \widetilde{\nabla} R = \widetilde{\nabla} S = 0$, where $\widetilde{\nabla} $ is  (see \cite{TriVan}) the connection $\widetilde{\nabla}  = \nabla -S$, $\nabla$ being the Levi-Civita connection of $g$. The condition $\widetilde{\nabla} g = 0$ is equivalent to $S_{XYZ} = - S_{XZY}$, where $S_{XYZ} = g(S_{X}Y,Z)$.

Ambrose and Singer \cite{AmbSin} gave the following characterization for homogeneous Riemannian manifolds: {\em A connected, simply-connected and complete Riemannian manifold $(M,g)$ is homogeneous if and only if it admits a homogeneous structure} $S$. Furthermore, considering  $M=G/H$ as a reductive homogeneous manifold with reductive decomposition $\fg = \fh \oplus  \fm$,  then $\widetilde{\nabla}  = \nabla -S$ is the canonical connection on $M$ with respect to  the given  reductive   decomposition (see \cite{TriVan}).

Let $(V , \langle\cdot,\cdot\rangle)$ be an $n$-dimensional Euclidean vector space. Tricerri and Vanhecke considered  in \cite{TriVan} the vector space $\mathcal{T}(V)$  of tensors of type $(1,2)$, or equivalently (using the inner product $\langle\cdot,\cdot\rangle$) of type $(0,3)$ satisfying the same algebraic symmetry that a  homogeneous structure, that is,
\[
\mathcal{T}(V ) = \{S\in \otimes^{3}V ^{*}\,:\,  S_{XYZ} = -S_{XZY},\; X,Y,Z\in V \}.
\]
They studied the decomposition of $\mathcal{T}(V )$ into invariant and irreducible components $\mathcal{T}_{i}(V )$, $i=1,2,3$, under the action of the orthogonal group $\Lie{O}(n)$. The tensors $S$ in the class $\mathcal{T}_{1}$ are those for which there exists $\xi\in V$ such that $S_{X}Y = \langle X,Y\rangle\xi - \langle\xi,Y\rangle X$. The tensors $S$ in the class $\mathcal{T}_1\oplus\mathcal{T}_2$ are those satisfying that the cyclic sum
$\mathfrak{S}_{XYZ} S_{XYZ}$ vanishes. Those of type $\mathcal{T}_{2}$ are the ones satisfying moreover $c_{12}(S)(X) = 0$, for all $X\in V$, where $c_{12}(S)(X) = \sum_{i}S_{e_{i}e_{i}X}$ for an arbitrary orthonormal basis $\{e_{i}\}$ of $V$. Those of type $\mathcal{T}_3$ are the ones in $\mathcal{T}(V )$ satisfying $S_{XYZ} = -S_{YXZ}$.

A homogeneous structure $S$ on $(M,g)$ is said to be of type ${\mathscr T}$, ${\mathscr T}(V)$ being one of the eight invariant subspaces of ${\mathcal T}(V),$ if $S_{p}\in {\mathscr T}_{p}(T_{p}M)$ for all $p\in M$.

The torsion $\wT$ of the connection $\wn$ is completely determined by the homogeneous structure $S$ as follows:
\begin{equation}\label{TS}
\wT_{XYZ} = S_{YXZ} - S_{XYZ},
\end{equation}
where $\wT_{XYZ} = g(\wT_{X}Y,Z)$. When it be necessary to refer to the metric, we shall say that $\wT$, as a tensor field of type $(0,3)$, is the $g$-torsion of $\wn$. Hence,
\[
2\Cyclic_{XYZ} S_{XYZ}  = -\Cyclic_{XYZ}\wT_{XYZ}\quad\mbox{\rm and}\quad c_{12}(S)(X) = \tr \;\wT_{X}.
\]
Conversely, since $\wn$ is a metric connection, $S$ can be expressed in terms of $\wT$ as (\cite[p.\ 83]{Kow})
\begin{equation}\label{ST}
2 S_{XYZ} = \wT_{YXZ} + \wT_{YZX} + \wT_{XZY}.
\end{equation}
Then, from (\ref{TS}) and (\ref{ST}) and according with the terminology in \cite[p.\ 222]{Puh}, the $g$-torsion $\wT$ is said to be
\begin{enumerate}
\item[$\bullet$] {\em vectorial} if there exists a vector field $\xi$ such that
\[
\wT_{XYZ} = g(X,Z)g(\xi,Y) - g(Y,Z)g(X,\xi),
\]
or equivalently if $S\in {\mathcal T}_{1};$ \smallskip
\item[$\bullet$] {\em cyclic} if $\cyclic_{XYZ}\wT_{XYZ} = 0,$ or equivalently if $S\in \mathcal{T}_1\oplus\mathcal{T}_2;$ \smallskip
\item[$\bullet$] {\em traceless} if $\tr \;\wT_{X} = 0,$ or equivalently if $S\in \mathcal{T}_2\oplus\mathcal{T}_3;$ \smallskip
\item[$\bullet$] {\em traceless cyclic} if $\wT$ is traceless and cyclic, or equivalently if $S\in{\mathcal T}_{2};$ \smallskip
\item[$\bullet$] {\em totally skew-symmetric} if $\wT_{XYZ} = -\wT_{XZY},$ or equivalently if $S\in {\mathcal T}_{3}.$
\end{enumerate}

For the vectorial case, taking $\varphi(X) = g(\xi,X),$ one gets that $\wT$ is vectorial if and only if there exists a one-form $\varphi$ on $M$ such that
\[
\wT_{X}Y = \varphi(Y)X - \varphi(X)Y.
\]
Note that the properties of being vectorial or traceless do not depend on the metric $g$.  

\section{Cyclic metric Lie groups}
\label{sectre}
\setcounter{equation}{0}

Let $G$ be a connected Lie  group and  let $\fg$ be  its Lie algebra. It is well known (see \cite{Nom}) that there exists a one-to-one correspondence between the set of left-invariant affine connections  on $G$  and the set of bilinear functions $\alpha$ on $\fg\times \fg$ with values in $\fg$.  The correspondence is given by $\alpha(X,Y) = \nabla_{X}Y$, for all $X,Y\in \fg$. For $\alpha =0$ one gets the canonical connection $\wn$, called the
$(-)$-connection of Cartan-Schouten,  i.e., $\wn_{X}Y = 0$, for all $X,Y\in \fg$. Its torsion tensor $\wT$ is given by $\wT_{X}Y = -[X,Y]$ and its curvature tensor $\wR$ is identically zero. Let $\nabla$ be the Levi-Civita connection of a left-invariant metric $g$ on $G$. Then the homogeneous structure $S = \nabla -\wn$ is left-invariant and it is determined by $S_{X}Y = \nabla_{X}Y$, for $X,Y\in \fg$. Using the Koszul formula, or directly from \eqref{ST}, we have 
 \begin{equation}
\label{cc}
 2\langle S_{X}Y,Z\rangle = \langle [X,Y],Z\rangle - \langle [Y,Z],X\rangle + \langle [Z,X],Y\rangle,
 \end{equation}
 for all $X,Y,Z\in \fg$, where $\langle\cdot,\cdot\rangle$ denotes the inner product on $\fg$ corresponding with $g$. Let $U\colon \fg\times \fg\to \fg$ be the symmetric bilinear mapping defined by
\[
2\langle U(X,Y),Z\rangle = \langle[Z,X],Y\rangle + \langle[Z,Y],X\rangle.
\]
Then $S$ is characterized by
\begin{equation}
\label{connection}
S_{X}Y = \frac{1}{2}[X,Y] + U(X,Y).
\end{equation}
Note that $\langle S_{X}Y,Z\rangle + \langle Y,S_{X}Z\rangle = 0$ and $S =0$ if and only if $\fg$ is abelian, or equivalently, because $\wT =0$, if $\wn$ is a flat connection. The curvature tensor $R$ of the Levi-Civita connection is given by
\begin{equation}\label{Rcur}
R(X,Y) = S_{[X,Y]} - [S_{X},S_{Y}],\qquad X,Y\in \fg.
\end{equation}

A metric $g$ on $G$ which is both left- and right-invariant is called biinvariant. This is equivalent to $\langle\cdot,\cdot\rangle$ being
$\Ad(G)$-invariant or also, under our hypothesis of connectedness, to the $g$-torsion $\wT$ being totally skew-symmetric. Hence, $S\in \cT_{3}$, or equivalently, $U=0$. Then, for any biinvariant metric, $S_{X} = \frac{1}{2}\ad_{X}$ and one gets
\[
R(X,Y)Z = \frac{1}{4}[[X,Y],Z],\quad \kappa(X,Y) =\frac{1}{4}\|[X,Y]\|^{2},
\]
where $\kappa$ is the curvature function $\kappa\colon \fg\times \fg\to \fg$, $\kappa(X,Y) =\langle R(X,Y)X,Y\rangle$. So, the sectional curvature $K$  is always nonnegative and there exists a section $\pi = \R\{X,Y\}$ such that $K(\pi) = 0$ if and only if $[X,Y] = 0$. The Ricci tensor $\mathrm{Ric}$ on $\fg$ is given by
\begin{equation}\label{Ricc}
\mathrm{Ric} = -\frac{1}{4}B,
\end{equation}
where $B$ is the Killing form of $G$.

\begin{definition} {\rm A left-invariant metric $g$ on a Lie group $G$ (or a metric Lie group $(G,g)$) is said to be {\it vectorial, cyclic\/} or {\it traceless cyclic\/} if the $g$-torsion of $\wn$ so is.}
 \end{definition}

Then a metric Lie group $(G,g)$ is vectorial if the bracket product on ${\mathfrak g}$ satisfies 
\[
\label{l}
[X,Y] = \varphi(X)Y -\varphi(Y)X,
\]
where $\varphi\in {\mathfrak g}^{*},$ ${\mathfrak g}^{*}$ being the dual space of ${\mathfrak g}$. Because $\tr\;\wT_{X} = -\tr \ad_{X},$ for all $X\in {\mathfrak g}$, it follows that $\wT$ is traceless if and only if the Lie group $G$ is unimodular. Hence, $(G,g)$ is cyclic if and only if
\[ 
\Cyclic_{XYZ}\langle [X,Y],Z\rangle = 0
\] 
and it is traceless cyclic if moreover $G$ is unimodular. 

It has been proved in \cite[Theorem 5.2]{TriVan} that a connected, simply-connec\-ted and complete Riemannian manifold
admits a non-trivial homogeneous structure $S\in\cT_{1}$ if and only it is isometric to the
real hyperbolic space. Moreover, for any given dimension $n$, this homogeneous structure corresponds to the representation of the real hyperbolic space as the solvable
  Lie group
\[        H^n(c)=\left\{\left(
               \begin{array}{cc}
                       \E^{c\,u} I_{n-1}   &   x \\
                        0      &   1
               \end{array}        \right)
                               \in GL(n,\R) :
u\in\R,\;x\in\R^{n-1} \right\}\, ,
 \]
equipped with a suitable left-invariant metric with constant sectional curvature $-c^{2}$. See Example 5.6 for more details. Hence, arguing as in the proof of Theorem 5.2 in \cite{TriVan}, we have the following result.  

\begin{proposition} 
\label{prop32}
Any simply-connected, nonabelian vectorial metric Lie group is isometrically isomorphic to $H^{n}(c),$ for some $c\neq 0$.
\end{proposition}

For cyclic left-invariant metrics, using (\ref{cc}),  one has that
 \begin{equation}\label{Sco}
 S_{XYZ} = \langle \nabla_{X}Y,Z\rangle = \langle [X,Y],Z\rangle.
 \end{equation} 
Then it follows that
\begin{equation}
\label{curvature}
\kappa(X,Y) = -\|[X,Y]\|^{2} + \langle S_{X}Y,S_{Y}X\rangle - \langle S_{X}X,S_{Y}Y\rangle. 
\end{equation} 

\begin{proposition}
If\, $X$ belongs to the center of the Lie algebra $\fg$ of a cyclic metric Lie group  then $\kappa(X,Y) = 0$ for all $Y\in \fg$.
\end{proposition}
\begin{proof}
By \eqref{Sco}, $S_XY = 0$, and the result then follows from (\ref{curvature}).
\end{proof}

From (\ref{Rcur}), any left-invariant metric on an abelian Lie group is flat, i.e., its Riemannian sectional curvature vanish. Next,  we prove  that the converse holds for the cyclic left-invariant case.

\begin{proposition}\label{pA}
A connected cyclic metric Lie group is flat if and only if it is abelian.
\end{proposition}
\begin{proof}
The Lie algebra ${\mathfrak g}$ of a Lie group equipped with a flat left-invariant metric splits as an orthogonal direct sum ${\mathfrak g} = {\mathfrak b} \oplus {\mathfrak u}$ such that ${\mathfrak b}$ is an abelian subalgebra, ${\mathfrak u}$ is an abelian ideal and the linear transformation
 $\ad_{B}$  is skew-symmetric for every $B\in {\mathfrak b}$ (see \cite[Theorem   1.5]{Mil}). But using that the metric is cyclic left-invariant, it follows that $\ad_{B}$ is also selfadjoint. Hence, one has $\ad_{B}(X) = 0,$ for all $X\in {\mathfrak g},$ and so the Lie group is abelian. 
\end{proof}

\begin{proposition}
\label{behav}
Let $G$ be a nonabelian cyclic metric Lie group. We have: \smallskip

\noindent {\rm (i)} If $G$ is solvable then it has strictly negative scalar curvature. \smallskip

\noindent {\rm (ii)}  If $G$ is unimodular then there exist positive sectional curvatures. If moreover it is solvable, it has both positive and negative sectional curvatures. \smallskip

\noindent {\rm (iii)}  If $G$ is not unimodular there exist negative sectional curvatures. \smallskip
\end{proposition}
\begin{proof} Properties (i) and (ii) follow from Proposition \ref{pA} together with Theorem   1.6, Theorem   3.1 and Corollary   3.2 in \cite{Mil}. If $G$ is not unimodular, its unimodular kernel $\fu$, that is,
\[
\fu = \{X\in \fg \,:\, \tr \ad_{X} = 0\},
\]
is an ideal of codimension one. Let $W$ be a unit vector orthogonal to $\fu$. Then
\[
\nabla_{W}W = 0,\quad \nabla_{W}X = \frac{1}{2}( \ad_{W} -\ad_{W}^{\,t})X,
\]
for all $X\in \fu$ (see \cite{Mil} for the details). Hence, if the inner product $\langle\cdot,\cdot\rangle$ on the Lie algebra ${\mathfrak g}$ of $G$ is cyclic left-invariant, it follows from (\ref{Sco}) that $ \ad_{W\mid {\mathfrak u}}$ is a  selfadjoint operator and $S_{W} =0$. So, from (\ref{curvature}), $K(W,X) = -\|[W,X]\|^{2}$, for all $X\in {\mathfrak u}$. Since there exists $X\in {\mathfrak u}$ such that $ \ad_{W}X \neq 0$,  one gets $K(W,X)< 0$.  This  proves (iii).
\end{proof}

\section{Semisimple cyclic metric Lie groups}
\label{seccua}
\setcounter{equation}{0}

The Killing form $B$ of a semisimple Lie group $G$ provides a biinvariant scalar product making $G$ a biinvariant  pseudo-Riemannian Lie group, which is, by (\ref{Ricc}), an Einstein manifold. When $G$ is moreover compact, $B$ is negative definite and $\langle\cdot,\cdot\rangle = -B(\cdot,\cdot)$ determines a biinvariant (Riemannian) metric.

For a general Lie group $G$ with a biinvariant Riemannian metric determined by an inner product $\langle\cdot,\cdot\rangle$ on its Lie algebra $\fg$, the orthogonal complement of any ideal in $\fg$ is itself an ideal. So $\fg$ can be expressed as an orthogonal direct sum
\[
\fg = {\mathcal Z}(\fg) \oplus \fg_{1}\oplus \dots \oplus \fg_{\ell},
\]
where its center ${\mathcal Z}(\fg)$ is isomorphic to $\R ^{k}$ for some $k$, and $\fg_{1},\dotsc ,\fg_{\ell}$ are compact simple ideals (see \cite{Mil} for details). Then $\langle\cdot,\cdot\rangle$ on $\fg$ is of the form
\[
\langle\cdot,\cdot\rangle = \langle\cdot,\cdot\rangle_{0} + \beta_{1}B_{1} +\dots + \beta_{\ell}B_{\ell},
\]
where $\langle\cdot,\cdot\rangle_{0}$ is the standard inner product on $\R ^{k}$, $B_{i}$, $i = 1,\dotsc,\ell$, is the restriction of the Killing form $B$ to $\fg_{i}\times \fg_{i}$ and $\beta_{i}<0$. Hence, taking into account that ${\mathcal Z}(\fg) = 0$ for  semisimple Lie groups, it follows that a  semisimple Lie group is biinvariant if and only if it is compact.

By contrast, we have in the cyclic left-invariant case the next result. 

\begin{lemma}
\label{cslg}
Every semisimple cyclic metric Lie group is not compact. Moreover, each cyclic
left-invariant metric is determined by an inner product $\langle\cdot,\cdot\rangle$ on its Lie algebra making orthogonal
an arbitrary $B$-orthonormal basis $\{e_{i}\}_{i=1}^{n}$ and $\langle e_{i},e_{i}\rangle = \varepsilon_{i}\lambda_{\,i},$
$i=1,\dotsc ,n,$ where $\varepsilon_{i} = B(e_{i},e_{i})$ and the $\lambda_{\,i}$'s  satisfy
\begin{equation}
\label{cc1}
c_{ij}^{\,k}(\lambda_{\,i} + \lambda_{j} + \lambda_{k}) = 0, \qquad
1\leq i<j<k\leq n,
\end{equation}
$c_{ij}^{\,k}$ being the structure constants given by $[e_{i},e_{j}]
= \sum_{k=1}^{n}c_{ij}^{\,k}e_{k}.$
\end{lemma}

\begin{proof} Let $G$ be a  semisimple Lie group and let $\langle\cdot,\cdot\rangle$ be a cyclic left-invariant inner product on the Lie algebra $\fg$
of $G$. Then $\langle\cdot,\cdot\rangle$ can be expressed in terms of its Killing form $B$ as
\[
\langle X,Y\rangle = B(QX,Y),\qquad X,Y\in \fg,
\]
where $Q$ is some selfadjoint operator on $\fg$. Let $\{e_{1},\dotsc ,e_{n}\}$ be a $B$-or\-tho\-nor\-mal basis of eigenvectors of $Q$ and let $\lambda_{1},\dotsc ,\lambda_{n}$ be the corresponding eigenvalues. Since $\langle\cdot,\cdot\rangle$ is nondegenerate, $\lambda_{\,i}\neq 0$ for all $i\in \{1,\dotsc,n\}$. Then $\{e_{1},\dotsc,e_{n}\}$ is an orthogonal basis with respect to $\langle\cdot,\cdot\rangle$ satisfying $\langle e_{i},e_{i}\rangle = \va_{i}\lambda_{\,i}$, $i = 1,\dotsc ,n$, where $\va_{i} = B(e_{i},e_{i})$; with $\va_{i} = -1$ if $\lambda_{\,i} < 0$ and $\va_{i} = 1$ if $\lambda_{\,i} > 0$. Hence, the structure  constants  $c_{ij}^{\,k}$ are given by $c_{ij}^{\,k} = \va_{k}B([e_{i},e_{j}],e_{k})$. Putting $\bar{c}_{ij}^{\,k} = \va_{k}c_{ij}^{\,k}$, one gets
\begin{equation}\label{constants}
\bar{c}_{ij}^{k} = -\bar{c}_{ji}^{\,k},\quad \bar{c}_{ij}^{\,k} = -\bar{c}_{ik}^{j}.
\end{equation}
The condition of cyclic left-invariance  for $\langle\cdot,\cdot\rangle$ is equivalent to 
\[
\lambda_{k}\bar{c}_{ij}^{\,k} +  \lambda_{j}\bar{c}_{ki}^{j} + \lambda_{\,i}\bar{c}_{jk}^{i} = 0.
\]
Hence, using (\ref{constants}), we have
\[
\bar{c}_{ij}^{\,k}(\lambda_{\,i} + \lambda_{j} + \lambda_{k}) = 0.
\]
As ${\mathcal Z}(\fg) = 0$, it follows that for each $i\in \{1,\dotsc ,n\}$, $\bar{c}_{ij}^{\,k}$ is different from zero for some $j$ and $k$, and so, $\lambda_{\,i} = -(\lambda_{j} + \lambda_{k})$. This  implies that not all the $\lambda_{\,i}$'s have the same sign and so, $B$ is not definite. Hence, $G$ cannot be compact.

Every cyclic inner product $\langle\cdot,\cdot\rangle$ on ${\mathfrak g}$ is then obtained taking as $Q$ the operator given
by $Qe_{i} = \lambda_{\,i}e_{i},$ where $\{e_{i}\}_{i=1}^{n}$ is an arbitrary $B$-orthonormal basis and the $\lambda_{\,i}$'s satisfy
our hypothesis. This proves the result.
\end{proof}

It is well known that every compact Lie algebra ${\mathfrak g}$ is the direct sum ${\mathfrak g} = {\mathcal Z}({\mathfrak g}) \oplus [{\mathfrak g},{\mathfrak g}]$, where the ideal $[{\mathfrak g},{\mathfrak g}]$ is compact and  semisimple. Since the restriction of a cyclic-invariant inner product to $[{\mathfrak g},{\mathfrak g}]$ must be again cyclic-invariant,  from Lemma \ref{cslg} the following result follows. 

\begin{theorem}
\label{theo34}
Every  nonabelian cyclic metric Lie group is  not  compact.
\end{theorem}

Since $\SU(2)$ is the only connected, simply-connected Lie group which admits a left-invariant metric of strictly positive sectional curvature (see \cite{Wal}), one directly obtains
\begin{corollary}
\label{noposi}
There is no cyclic metric Lie group  with strictly positive sectional curvature.
\end{corollary}

For  the simple  case we have the following result. 
\begin{theorem}
\label{simple}
The universal covering group $\widetilde{\SL(2,\R)}$ of $\SL(2,\R)$  is the only connected, simply-connected simple real cyclic metric Lie group.
\end{theorem}
\begin{proof}
Let $\fg_\C$ be a simple Lie algebra over $\C$ and let $\sigma$ be an involutive automorphism of a compact real form $\fg$ of $\fg_\C$. Then $\fg = \fu \oplus \fp$, where $\fu$ and $\fp$ denote the eigenspaces of $\sigma$ with eigenvalues $+1$ and $-1$, respectively. Let $(\fg^{*},\sigma^{*})$ be the dual orthogonal symmetric Lie algebra of $(\fg,\sigma)$, $\fg^{*}$ being the subspace of $\fg_{\C}$ defined by the Cartan decomposition $\fg^{*} = \fu \oplus \I\fp$.

If $\fg^*$ admits a cyclic left-invariant inner product, since the Killing form $B$ is strictly negative definite on the maximal compactly embedded subalgebra $\fu$,
it follows from (\ref{cc1}) that $\fu$ must be  abelian.  Now, according for instance to \cite[pp.\ 695--718]{Kna}, the only simple real Lie algebras
whose maximal compact Lie subalgebra is  abelian  are
$\fsl(2,\R)$, $\fsu(1,1)$, $\fso(2,1)$ and $\fsp(1,\R)$, which are mutually isomorphic
and define the connected, simply-connected simple Lie group $\widetilde{\SL(2,\R)}$.
\end{proof}

\section{Solvable cyclic metric Lie groups}
\label{seccin}
\setcounter{equation}{0}

Let $G$ be an $n$-dimensional solvable Lie group. Then its Lie algebra ${\mathfrak g}$ satisfies the chain condition or, equivalently, there exists a sequence
\[
{\mathfrak g} = {\mathfrak g}_{0} \supset {\mathfrak g}_{1}\supset \dots \supset {\mathfrak g}_{n-1} \supset{\mathfrak g}_{n} = \{0\}\,,
\] 
where ${\mathfrak g}_{r}$ is an ideal in ${\mathfrak g}_{r-1}$ of codimension $1$, $1\leq r\leq n$. Given an inner product $\langle\cdot,\cdot\rangle$ on ${\mathfrak g}$ we construct an orthonormal basis $\{e_{1},\dotsc ,e_{n}\}$, called {\em adapted to} $({\mathfrak g},\langle\cdot,\cdot\rangle)$, such that ${\mathfrak g}_{n-i} = \R\{e_{1},\dotsc ,e_{i}\}$, $i=1,\dotsc ,n$. Then the structure constants $c_{ij}^{k} = \langle[e_{i},e_{j}],e_{k}\rangle$ of $G$ with respect  to this  basis satisfy
\begin{equation}
\label{solv}
c_{ij}^{k} = 0 \quad \mbox{for}\;\,k  \geq {\rm max}\{i,j\}.
\end{equation}
Hence, we have the next result.
\begin{proposition}\label{pco}
The following conditions are equivalent for an $n$-di\-men\-sio\-nal solvable Lie group $G$. \smallskip

\noindent {\rm (i)}  $G$ is a cyclic metric Lie group. \smallskip

\noindent {\rm (ii)}  $c_{ik}^{j} = c_{jk}^{i}$, for all $ 1\leq i<j<k\leq n$. \smallskip

\noindent {\rm (iii)}  $ \ad_{e_{i}}$ is selfadjoint on ${\mathfrak g}_{n-i}$, for all $i = 1\dotsc,n$.
\end{proposition}

\begin{remark} The property (\ref{solv}) together with Proposition \ref{pco} (ii) determine completely the structure constants of any solvable cyclic metric Lie group. Then they can be useful to give examples and in fact to obtain classifications for low dimensions.
\end{remark}

\smallskip

Let $G_{1}$ and $G_{2}$ be Lie groups equipped with left-invariant metrics determined by inner products $\langle\cdot,\cdot\rangle_{1}$ and $\langle\cdot,\cdot\rangle_{2}$ on their corresponding Lie algebras ${\mathfrak g}_{1}$ and ${\mathfrak g}_{2}$. For each homomorphism $\pi$ of $G_{1}$ into the group $\Aut(G_{2})$  of automorphisms of $G_{2}$, consider the semidirect product $G_{1}\ltimes_{\pi}G_{2}$ equipped with the Riemannian product on $G_{1}\times G_{2}$. Its Lie algebra is the semidirect sum $\fg_{1}+_{\pi_{*}}\fg_{2}$ of the Lie algebras $\fg_{1}$ and $\fg_{2}$. The differential $\pi_{*}$ of $\pi$ is a Lie algebra homomorphism $\fg_{1}\to \mathrm{Der}(\fg_{2})$ and the bracket product on ${\mathfrak g}_{1}+_{\pi_*}{\mathfrak g}_{2}$ is given by
\[
[(X_{1},X_{2}),(Y_{1},Y_{2})] = \bigl([X_{1},Y_{1}],\, [X_{2},Y_{2}] + \pi_{*}(X_{1})(Y_{2}) - \pi_{*}(Y_{1})(X_{2})\bigr),
\]
for all $X_{i}, Y_{i}\in {\mathfrak g}_{i}$, $i = 1,2$. Since the inner product $\langle\cdot,\cdot\rangle = \langle\cdot,\cdot\rangle_{1} + \langle\cdot,\cdot\rangle_{2}$ on ${\mathfrak g}_{1}+_{\pi_*}{\mathfrak g}_{2}$ satisfies
$$
\begin{array}{lcl}
\langle [(X_{1},X_{2}),(Y_{1},Y_{2})],(Z_{1},Z_{2})\rangle & = & \langle [X_{1},Y_{1}],Z_{1}\rangle_{1} + \langle [X_{2},Y_{2}],Z_{2}\rangle_{2}\\[0.4pc]
& & \;\,+\, \langle \pi_{*}(X_{1})(Y_{2}) - \pi_{*}(Y_{1})(X_{2}),\,Z_{2}\rangle_{2}.
\end{array}
$$
one gets the following  result.
\begin{lemma}
\label{semidi} Let $G_1$ and $G_2$ be Lie groups, each of them equipped with a left-invariant metric. The Riemannian product metric  on $G_{1}\times G_{2}$ defines a cyclic left-invariant metric on the semidirect product $G_{1}\ltimes_{\pi}G_{2}$ if and only if the left-invariant metrics on $G_{1}$ and $G_{2}$ are cyclic and the derivation $\pi_{*}(X_{1})$ on ${\mathfrak g}_{2}$, for each $X_{1}\in {\mathfrak g}_{1}$,  is selfadjoint with respect to $\langle\cdot,\cdot\rangle_{2}$.
\end{lemma}

Next, we show that any solvable cyclic metric Lie group can be expressed as an orthogonal semidirect product. Concretely, we have the next result. 
\begin{proposition}
\label{psplit}
Any nontrivial connected, simply-connected solvable cyclic metric Lie group  decomposes into an orthogonal semidirect product $\R\ltimes_{\pi}N$, where $N$ is a unimodular normal Lie subgroup of codimension one and $\pi_{*}(\D/\D t)$ $= \ad_{\D/\D t}$  is a selfadjoint derivation on the Lie algebra of $N$.
\end{proposition}
\begin{proof} Any (real) nonunimodular metric Lie algebra $(\fg,\langle\cdot,\cdot\rangle)$,  not necessarily  solvable, can be expressed as an orthogonal semidirect sum $\fg = \R W +_{\ad_{W}}  \fu$, where $\fu$ is its unimodular kernel. For the unimodular case, we remark that if $(\fg,\langle\cdot,\cdot\rangle)$ is a nontrivial solvable Lie algebra equipped with a  cyclic left-invariant  inner product and $\{e_{1},\dotsc ,e_{n}\}$ is an adapted orthonormal basis, then ${\mathfrak g}$ splits into the orthogonal semidirect sum of $\R e_{n}$ and the ideal ${\mathfrak g}_{1} = \R\{e_{1},\dotsc,e_{n-1}\}$, the restriction of $\langle\cdot,\cdot\rangle$ to ${\mathfrak g}_{1}$ is also  cyclic left-invariant   and $\ad_{e_{n}}$  is selfadjoint on ${\mathfrak g}_{1}$. Hence, the  connected and simply-connected Lie group with Lie algebra $\fg$ satisfies the conditions in the statement. 
\end{proof}

\begin{proposition}
\label{nabnilp}
Nonabelian  nilpotent Lie groups do not admit cyclic left-invariant metrics. 
\end{proposition}
\begin{proof}
Consider the central descending series of the Lie algebra ${\mathfrak g}$ of a Lie group $G$,
\[
{\mathfrak g} = {\mathcal C}^{0} {\mathfrak g} \supset {\mathcal C}^{1}{\mathfrak g} \supset \dots \supset {\mathcal C}^{m-1}{\mathfrak g} \supset {\mathcal C}^{m}{\mathfrak g}= \{0\},
\]
defined by ${\mathcal C}^{p+1}{\mathfrak g} = [{\mathfrak g}, {\mathcal C}^{p}{\mathfrak g}]$, $p = 0,1, \dotsc, m-1$. Then ${\mathcal C}^{m-1}{\mathfrak g}\neq \{0\}$, it is in the center ${\mathcal Z}({\mathfrak g})$ of ${\mathfrak g}$ and, if  $m\geq 2$, then  ${\mathcal C}^{m-1}{\mathfrak g}\subset [{\mathfrak g},{\mathfrak g}]$. Because ${\mathcal Z}({\mathfrak g})$ and the derived algebra $[{\mathfrak g},{\mathfrak g}]$ of ${\mathfrak g}$ are orthogonal with respect to any cyclic left-invariant metric, $m$ must be $1$ and  this  implies that ${\mathfrak g}$ is abelian.
\end{proof}

\begin{example} The metric solvable Lie group $G^{\,n}(\alpha_1,\alpha_{\,2},\ldots,\alpha_{\,n-1})$.   
\label{solvable-gen-1}

For $(\alpha_1,\alpha_{\,2},\ldots,\alpha_{\,n-1})\in\R^{n-1}\setminus\{(0,\ldots,0)\}$,
let $\fg=\fg(\alpha_1,\alpha_{\,2},\ldots,\alpha_{\,n-1})$ be the $n$-dimensional metric Lie algebra generated by the
  basis $\{e_1,\ldots,e_{n}\}$
 with Lie brackets
\[
  [e_{n},e_i]=\alpha_{\,i} e_i, \;\ 1\leq i\leq n-1\,; \quad [e_i,e_j]=0, \;\ 1\leq i<j\leq n-1, 
\]
and equipped with the inner product $\langle\cdot,\cdot\rangle$ for which
$\{e_1,\ldots,e_{n}\}$ is orthonormal. Then $\fg$ can be identified with the
orthonormal semidirect sum $\R\{e_{n}\}  +_{\ad_{e_n}}  \R\{e_1,\ldots,e_{n-1}\}$ under the adjoint
representation. It is unimodular if and only if $\sum_{i=1}^{n-1}\alpha_{\,i} = 0$.

The connected and simply-connected Lie group $G$ generated by $\fg$ must
be isomorphic to the orthogonal semidirect product $\R\ltimes_\pi\R^{n-1}$ under the action
$\pi\colon\R\rightarrow\Aut(\R^{n-1})$ given by 
$\pi(t) = \diag\left(\E^{\alpha_1 t}, \E^{\alpha_{\,2} t}, \dots, \E^{\alpha_{\,n-1} t}\right)$. 
Because the  structure constants  $c_{ij}^{k}$ with respect to $\{e_{i}\}_{i=1}^{n}$  satisfy  \eqref{solv}, $G$ is a solvable metric Lie group with $\{e_{i}\}_{i=1}^{n}$ as an adapted orthonormal basis. Moreover, from Proposition~\ref{pco} or using Lemma \ref{semidi}, $\langle\cdot,\cdot\rangle$ determines a cyclic left-invariant metric.  The group $G$  can be described as 
the group $G^{\,n}(\alpha_1,\alpha_{\,2},\dotsc,\alpha_{\,n-1})$ of matrices of the form
\[
  \left(
   \begin{array}{ccccc}
\E^{\alpha_1 u} & 0 & \cdots & 0 & x_1\\
\noalign{\smallskip}
0 & \E^{\alpha_{\,2} u} & \cdots & 0 & x_2\\
\noalign{\smallskip}
\vdots & \vdots & \ddots & \vdots & \vdots\\
\noalign{\smallskip}
0& 0& \ldots & \E^{\alpha_{\,n-1} u} & x_{n-1} \\
\noalign{\smallskip}
0 & 0 &  \cdots & 0 & 1
\end{array}
\right).
\]
We can consider the global coordinate system $(u,x_1,\dotsc,x_{n-1})$ of this matrix group, and for each $A\in G^{\,n}(\alpha_1,\dotsc,\alpha_{\,n-1})$, we have
$u\comp L_{A} = u(A) + u$, $x_i\comp L_{A} = x_i(A)+\E^{\alpha_{\,i} u(A)}x_i$, $1\leq i\leq n-1$. 
Hence the left-invariant Riemannian metric $g$ on $G^{\,n}(\alpha_1,\alpha_{\,2}$, $\dotsc,\alpha_{\,n-1})$ 
defined by the inner product $\langle\cdot,\cdot\rangle$ on $\fg$ is
  \[
g = \D u^{2} + \sum_{i=1}^{n-1}\E^{-2\alpha_{\,i} u}\D x_i^2.
\]
The Levi-Civita connection $\na$ of $g$ is given by $\na_{e_{n}}e_{n}=\na_{e_{n}}e_i=0$, $\na_{e_i}e_{n}=-\alpha_{\,i} e_i$, 
$\na_{e_i}e_j=\delta_{ij}\alpha_{\,i} e_{n}$, $1\leq i,j\leq n-1$.  
This implies that $V = e_{n}$ is a geodesic vector and, from \cite[Proposition 6.1]{GDV}, it is a harmonic vector field. Moreover, it defines a harmonic immersion into the unit tangent sphere of $G^{\,n}(\alpha_1,\alpha_{\,2},\dotsc,\alpha_{\,n-1})$ if and only if $\sum_{i=1}^{n-1}\alpha_{\,i}^{3}= 0$. The curvature tensor field satisfies
$R_{e_ie_{n}}e_i= -\alpha_{\,i}^2 e_{n}$, $R_{e_ie_j}e_i=-\alpha_{\,i}\alpha_j\,e_j$, $1\leq i\not= j \leq n-1$. 
Then $\{e_{1},\dotsc,e_{n}\}$ is a basis of eigenvectors for the Ricci tensor and the principal Ricci curvatures are
\begin{equation}\label{Ricci-Galphai}
  r(e_{n})=-\sum_{j=1}^{n-1}\alpha_j^2, \quad r(e_i)=-\alpha_{\,i} \sum_{j=1}^{n-1}\alpha_j, \quad 1\leq i\leq n-1.
\end{equation}
The scalar curvature is
\[
      s=-2\Big(\sum_{i=1}^{n-1}\alpha_{\,i}^{2} + \sum_{i<j}^{n-1}\alpha_{\,i}\alpha_{j}\Big).
\]
For the sectional curvatures of basic sections we have
$K(e_i,e_{n})= -\alpha_{\,i}^2$, $K(e_i,e_j)$ $=-\alpha_{\,i}\alpha_j$, $1\leq i\not= j \leq n-1$. 

If $G^{\,n}(\alpha_1,\dotsc,\alpha_{\,n-1})$ is unimodular, i.e., $\sum_{i=1}^{n-1}\alpha_{\,i} = 0$, there exist $i,j\in \{1,\dotsc,n-1\}$ such that $\alpha_{\,i}\alpha_{j}<0$. Then, according with Proposition \ref{behav}, one finds both positive and negative sectional curvatures. If $\alpha_1=\cdots=\alpha_{\,n-1}=\alpha\not=0$, the metric Lie group $G^{\,n}(\alpha_1,\ldots,\alpha_{\,n-1})$
gives the solvable description $H^{n}(\alpha)$ of  the $n$-dimensional hyperbolic space
with constant sectional curvature $-\alpha^{2}$.
\end{example}

\begin{example} The metric  solvable Lie group $H^{n+1}(\rho_1,\ldots,\rho_{n-1};\lambda_1,\ldots,\lambda_{n-2})$. 
\label{solvable-gen-2}

Following Proposition \ref{psplit}, we determine the connected, simply-connected, 
solvable cyclic metric Lie groups, which are nontrivial one-dimensional orthogonal extensions
of the unimodular Lie group $G^{\,n}(\alpha_1,\ldots,\alpha_{\,n-1})$, 
$(\alpha_1,\ldots,\alpha_{\,n-1}) \not= (0,\ldots,0)$, $\alpha_1+\dotsb+\alpha_{\,n-1}=0$. 

Let $(H,g)$ be one such extension. The corresponding metric Lie algebra $\fh$ is then an orthogonal
semidirect sum $\fh=\R\{e_0\} +_{\ad_{e_0}} \R\{e_1,\ldots,e_{n}\}$, where $\{e_0,e_1,\ldots,e_{n}\}$ is an orthonormal basis of $\fh$ such that 
$[e_{n},e_i]=\alpha_{\,i} e_i$, $1\leq i\leq n-1$, $[e_i,e_j]=0$, $1\leq i<j\leq n-1$, and $\ad_{e_0}$ must act on $\R\{e_1,\ldots,e_{n}\}$ as a  selfadjoint operator, because $g$ must be
cyclic left-invariant. In terms of $\{e_1,\ldots,e_{n}\}$ we can write $\ad_{e_0}= (a_i^j)_{1\leq  i,j\leq n}$,  $a_i^j=a_j^i$. Since $\ad_{e_0}$ must also act as a derivation on $\R\{e_1,\ldots,e_{n}\}$, applying $\ad_{e_0}$ to
$[e_i,e_j]$ one has $a_i^{n}\alpha_j=0$, $1\leq i\not=j\leq n-1$, and applying $\ad_{e_0}$ to $[e_{n},e_j]$, we get
$a_i^{n}\alpha_{\,i}=0$, $a_{n}^{n}\alpha_{\,i}=0$, $1\leq i\leq n-1$, 
  then $a_{n}^k=a_k^{n}=0$ for $1\leq k\leq n$, and hence $[e_0, e_{n}]=0$.
  Thus $\ad_{e_0}$ and $\ad_{e_{n}}$ are commuting operators of $\R\{e_1,\dotsc,e_{n-1}\}$,
  hence there exist an orthonormal basis $\{v_1,\dotsc$, $v_{n-1}\}$ of $\R\{e_1,\dotsc,e_{n-1}\}$
  which consists of eigenvectors for both $\ad_{e_0}$ and $\ad_{e_{n}}$. We put $u_0=e_0$,
  $v_0=e_{n}$, and  then $\{u_0,v_0,v_1,\ldots,v_{n-1}\}$ is an orthonormal basis of $\fh$
  such that 
  \[
   [u_0,v_i]=\rho_i v_i, \quad [v_0,v_i]=\lambda_{\,i} v_i, \qquad 1\leq i\leq n-1,
  \]
  for some $(\rho_1,\ldots,\rho_{n-1})\not=(0,\ldots,0)$ (because $\ad_{e_0}$ must act as a nontrivial
  operator  on $\R\{e_1,\ldots,e_{n-1}\}$), and $(\lambda_1,\ldots,\lambda_{n-1})\not=(0,\ldots,0)$,
  with $\lambda_1+\dotsb+\lambda_{n-1}=0$  (because $\textrm{tr}\, \ad_{e_{n}}=\alpha_1+\dotsb+\alpha_{\,n-1}=0$).
   Therefore, $\fh$ is an orthogonal semidirect sum of its abelian subalgebras $\R\{u_0,v_0\}$
   and $\R\{v_1,\ldots,v_{n-1}\}$, and the connected simply-connected Lie group
   $H=H^{n+1}(\rho_1$, $\ldots,\rho_{n-1};\lambda_1,\ldots,\lambda_{n-2})$ generated by $\fh$ is the
   semidirect product $\R^2\ltimes_\pi\R^{n-1}$ under the action
$\pi\colon\R^2\rightarrow\Aut(\R^n)$, given by 
$\pi(s,t) =  \diag(\E^{\,\rho_1 s+\lambda_1 t}, \dotsc , \E^{\,\rho_{n-1} s + \lambda_{n-1} t})$, 
where $\lambda_{n-1} = -(\lambda_1+\dotsb+\lambda_{n-2})$. Then $H$ can be described as $\R^{n+1}$ with the group operation
\begin{multline*}
(u,v,x_1,\dotsc,x_{n-1})\cdot(u',v',x_1',\dotsc,x'_{n-1}) \\ = (u+u',\,v+v',\, x_1+\E^{\,\rho_1 u+\lambda_1 v}x'_1,\dotsc, 
                                                          x_{n-1}+\E^{\,\rho_{n-1} u+\lambda_{n-1} v}x'_{n-1}),
\end{multline*}
where $(\rho_1,\ldots,\rho_{n-1})$ and $(\lambda_1,\ldots,\lambda_{n-1})$ are different from $(0,\ldots,0)$. 
The Lie group $H$ is unimodular if and only if $\sum_{i=1}^{n-1}\rho_{i} = 0$.

 If $(\rho_1,\ldots,\rho_{n-1})$ is not a multiple of
  $(\lambda_1,\ldots,\lambda_{n-1})$ (in particular
  if $H$ is not unimodular) then $H$ is the group of matrices of the form 
   \begin{equation}  \label{H-rhoi-lambdai-2}
\left(
\begin{array}{cccc}
\E^{\,\rho_1 u+\lambda_1 v} &  \cdots & 0 & x_1\\
\noalign{\smallskip}
\vdots & \ddots & \vdots &  \vdots\\
\noalign{\smallskip}
0 & \cdots  & \E^{\,\rho_{n-1}u+\lambda_{n-1} v} & x_{n-1}\\
0 & \cdots  & 0              & 1
\end{array}
\right ).
  \end{equation}
In any case, since $(\lambda_1,\ldots,\lambda_{n-2})\not=(0,\dotsc,0)$, we can suppose
(reordering the vectors $v_1,\dotsc,v_{n-1}$ if necessary) that $\lambda_1\not=0$, and define
$\hat{u}_{0} = (1/\sqrt{\lambda_1^2+\rho_1^2})$ $(\lambda_1 u_{0}-\rho_1 v_{0})$, 
$\hat{v}_{0} = (1/\sqrt{\lambda_1^2+\rho_1^2})(\rho_1 u_{0}+\lambda_1 v_{0})$.
Then
\[
 [\hat{u}_0,v_i]=\sigma_i v_i, \quad [\hat{v}_0,v_i]=\mu_i v_i, \qquad 1\leq i \leq n-1,
\]
with
$\sigma_i= (\lambda_1\rho_i-\rho_1\lambda_{\,i})/\sqrt{\lambda_1^2+\rho_1^2}$, $\mu_i= (\rho_1\rho_i+\lambda_1\lambda_{\,i})/\sqrt{\lambda_1^2+\rho_1^2}$. We have $\sigma_1=0$, $\mu_1\not=0$, then $H$ can be described as the Lie group $\hat{H}^{n+1}(\sigma_2$, $\dotsc,\sigma_{n-1};\mu_1$, $\dotsc,\mu_{n-1})$ of matrices of the form
\begin{equation}    \label{H-sigmai-mui}
  \left(
   \begin{array}{ccccc}
\E^{\mu_1 v} & 0 & \cdots & 0 & x_1\\
\noalign{\smallskip}
0 & \E^{\sigma_2 u+\mu_2 v} & \cdots & 0 & x_2\\
\noalign{\smallskip}
\vdots & \vdots & \ddots & \vdots & \vdots\\
\noalign{\smallskip}
0& 0& \ldots & \E^{\sigma_{n-1} u+\mu_{n-1} v} & x_{n-1} \\
\noalign{\smallskip}
0 & 0 &  \cdots & 0 & 1
\end{array}
\right).
\end{equation}
In particular, if $H$ is unimodular, one gets that $\sigma_{n-1}=-\sum_{i=2}^{n-2}\sigma_i$,
$\mu_{n-1}=-\sum_{i=1}^{n-2}\mu_i$, and the family of such unimodular metric Lie groups depends on $2n-5$ parameters.

We consider the global coordinate system $(u,v,x_1,\ldots,x_{n-1})$ of the matrix Lie group $H^{n+1}(\rho_1,\dotsc,\rho_{n-1};\lambda_1,\dotsc,\lambda_{n-2})$. 
We put again $\lambda_{n-1}=-(\lambda_1$ $+\dotsb+\lambda_{n-2})$.
The generators of $\fh$ correspond to the left-invariant vector fields
$u_0= \partial/\partial u$, $v_0= \partial/\partial v$, $v_i=\E^{\,\rho_i u+\lambda_{\,i} v}\partial/\partial x_i$, $1\leq i\leq n-1$. 
The left-invariant Riemannian metric $g$ is given by
  \[
g = \D u^{2} + \D v^{2}+ \sum_{i=1}^{n-1}\E^{-2(\rho_i u+\lambda_{\,i} v)}\D x_i^2.
\]
In terms of the basis $\{u_0,v_0,v_1,\dotsc,v_{n-1}\}$ of $\fh$, the Levi-Civita connection $\na$ of $g$ is given by
$\nabla_{v_i}u_0=-\rho_i v_i$, $\nabla_{v_i}v_0=-\lambda_{\,i} v_i$,
$\nabla_{v_i}v_i=\rho_i u_0+\lambda_{\,i} v_0$, $1\leq i\leq n-1$, 
the other components being zero. The curvature tensor field satisfies
\begin{align*}
 &R_{u_{0}v_{i}}u_{0} = -\rho_{i}^{2}v_{i}, \quad R_{v_{0}v_{i}}v_{0} = -\lambda_{\,i}^2 v_{i},\quad R_{u_{0}v_{i}}v_{0} = R_{v_{0}v_{i}}u_{0} = -\lambda_{\,i}\rho_{i}v_{i},\\
\noalign{\smallskip}
 &R_{u_{0}v_{i}}v_{i} = \rho_{i}(\rho_{i}u_{0} + \lambda_{\,i}v_{0}), \quad R_{v_{0}v_{i}}v_{i} = \lambda_{\,i}(\rho_{i}u_{0} + \lambda_{\,i}v_{0}),\\
\noalign{\smallskip}
 &R_{v_{i}v_{j}}v_{i} = -(\lambda_{\,i}\lambda_{j} + \rho_{i}\rho_{j})v_{j}, \quad R_{v_{i}v_{j}}v_{j} = (\lambda_{\,i}\lambda_{j} + \rho_{i}\rho_{j})v_{i},
\end{align*}
for $1\leq i\not=j \leq n$, the other components being zero. Then the sectional curvatures of the basic sections are  given by
\begin{align*}
 &K(u_0,v_0)=0, \quad K(u_0,v_i)=-\rho_i^2,\quad K(v_0,v_i)=-\lambda_{\,i}^2,  &1\leq i\leq n-1,\\
\noalign{\smallskip}
 &K(v_i,v_j)=-(\rho_i\rho_j+\lambda_{\,i}\lambda_j), &1\leq i\not= j\leq n-1,
\end{align*}
and the nonvanishing components of the Ricci curvature by
$$
\begin{array}{l}
\Ric (u_{0},u_{0}) =  -\sum_{i=1}^{n-1}\rho_i^2, \quad  \Ric (v_0,v_{0})  = -\sum_{i=1}^{n-1}\lambda_{\,i}^2,\\[0.6pc]
\Ric (v_i,v_{i})  =  -\rho_i \sum_{j=1}^{n-1}\rho_j,   \quad \Ric (u_{0},v_{0}) = -\sum_{i=1}^{n-1}\lambda_{\,i}\rho_{i}.
\end{array}
$$
The scalar curvature is then
\[
      s=-\sum_{i=1}^{n-1}\lambda_{\,i}^2  - \sum_{i=1}^{n-1}\rho_i^2 -\Big(\sum_{i=1}^{n-1}\rho_i\Big)^2.
\]
If   $H^{n+1}(\rho_1,\dotsc,\rho_{n-1};\lambda_1,\dotsc,\lambda_{n-2})$  is unimodular, that is
$\sum_{i=1}^{n-1}\rho_i=0$, then the signature of the Ricci form is $(0,\stackrel{n-1}{\dots},0,0,-)$  or $(0,\stackrel{n-1}{\dots},0,-,-)$,
 depending on whether or not
 $(\lambda_1,\dotsc,\lambda_{n-2})$ and $(\rho_1,\dotsc,\rho_{n-2})$ are proportional, respectively.

 If $\rho_i=a\lambda_{\,i}$ for some $a\not=0$, and for each $i=1,\dotsc,n-2$,
 the metric Lie group $H^{n+1}(\rho_1,\dotsc,\rho_{n-1};$ $\lambda_1,\dotsc,\lambda_{n-2})$ is isometrically isomorphic
 to the direct product
 $G^{\,n}(\tilde{\lambda}_1,\dotsc,\tilde{\lambda}_{n-2},-\sum_{i=1}^{n-2}\tilde{\lambda}_{\,i})\times\R$, with
 $\tilde{\lambda}_{\,i}=\sqrt{1+a^2}\lambda_{\,i}$, $i=1,\dotsc,n-2$.
\end{example}

\section{Classification of connected, simply-connected cyclic metric Lie groups for dimensions $n\leq 5$}
\label{secsie}
\setcounter{equation}{0}

Because any homogeneous structure on a two-dimensional Riemannian manifold is of type ${\mathcal T}_{1}$ (see \cite[Theorem 3.1]{TriVan}), it follows that every  two-dimensional (nonabelian) metric Lie group is vectorial. Moreover, for the simply-connected case, Proposition \ref{prop32} implies that it is isometrically isomorphic to the Poincar\'e half-plane $H^{2}(c),$ for some $c\neq 0$.

Next, using the same process than in \cite[pp.\ 83-85]{TriVan}, we shall give the classification for dimension three. First, suppose that $(G,g)$ is a three-dimensional unimodular metric Lie group. Then there exists (cf.\ \cite{Mil}) an orthonormal basis $\{e_{1},e_{2},e_{3}\}$ of the Lie algebra $\fg$ of $G$ such that
\begin{equation}
\label{unia}
[e_{2},e_{3}]=\lambda_{1}e_{1},\quad [e_{3},e_{1}] = \lambda_{2}e_{2},\quad [e_{1},e_{2}] = \lambda_{3}e_{3},
\end{equation}
where $\lambda_{1}$, $\lambda_{2}$, $\lambda_{3}$ are constants. The Ricci tensor $\mathrm{Ric}$ is given by
\[
\mathrm{Ric} = 2\bigl(\mu_{2}\mu_{3}\,\theta^{1}\otimes \theta^{1} + \mu_{1}\mu_{3}\,\theta^{2}\otimes \theta^{2}
+ \mu_{1}\mu_{2}\,\theta^{3}\otimes \theta^{3}\bigr),
\]
where $\theta^{i}$, $i = 1,2,3$, are the $1$-forms dual to $e_{i}$, and $\mu_{i} = \frac{1}{2}(\lambda_{1} + \lambda_{2} + \lambda_{3}) -\lambda_{\,i}$, $i=1,2,3$. According with the signs of $\lambda_{\,i}$, $i=1,2,3$, we have six kinds of Lie algebras.

If $g$ is cyclic left-invariant, one gets from \eqref{unia} that $\lambda_{1} + \lambda_{2} + \lambda_{3} =0$. This implies that $\fg$
is isomorphic to either $\fsl(2,\R)$ with $\lambda_{3} = -(\lambda_{1}+\lambda_{2})$, $\lambda_{1},\lambda_{2}>0$, or to $\fe(1,1)$ with $\lambda_{1}=-\lambda_{3}>0$ and $\lambda_{2} =0$. In both cases, one gets $\mu_{i} = -\lambda_{\,i}$, $i=1,2,3$. Then the signature of $\mathrm{Ric}$ is $(-,-,+)$ for $\fsl(2,\R)$ and $(0,0,-)$ for $\fe(1,1)$. They give the simply-connected, traceless cyclic metric Lie groups $E(1,1)$ and the universal covering group $\widetilde{\SL(2,\R)}$ of $\SL(2,\R)$ described by Kowalski and Tricerri \cite[Theorem 2.1]{KowTri} in the cases (b) and (c), respectively.

The family of all cyclic left-invariant metrics on $\widetilde{\SL(2,\R)}$ can be explicitly determined as the inner products on $\mathfrak{so}(1,2)$,   under the natural identification with $\mathfrak{sl}(2,\R)$, making orthonormal the bases $\{e_{1},e_{2},e_{3}\}$ given by
\begin{equation}\label{basis}
\begin{array}{lcl}
e_{1} & = &  \sqrt{\lambda_{2}(\lambda_{1}+\lambda_{2})}(E_{12} + E_{21}),\\[0.4pc]
e_{2} & = & \sqrt{\lambda_{1}(\lambda_{1}+\lambda_{2})}(E_{13} + E_{31}),\\[0.4pc]
e_{3} & =  & \sqrt{\lambda_{1}\lambda_{2}}(E_{32}-E_{23}),
\end{array}
\end{equation}
for all $\lambda_{1},\lambda_{2} >0,$ where $E_{ij}$ denote the matrix on $\mathfrak{so}(1,2)$ with entry $1$ where the $i$ th row and the $j$ th column meet, all other entries being $0$.

Next, let $G$ be a nonunimodular Lie group. Because the unimodular kernel ${\fu}$ of $\fg$ is abelian, there exists an orthonormal basis $\{e_{1},e_{2},e_{3}\}$, with $\fu = \R\{e_{2},e_{3}\}$, such that
\[
[e_{1},e_{2}] = ae_{2} + be_{3},\quad [e_{1},e_{3}] = ce_{2} + de_{3},\quad [e_{2},e_{3}] = 0,
\]
where $a,b,c,d$ are real constants such that $a +d \neq 0$. This implies that $\fg$ is solvable, even for the unimodular case $a+d =0$. Moreover, if the metric $g$ on $G$ is cyclic left-invariant it follows that $\ad_{e_{1}}$ must be  selfadjoint on $\fu$. Then we can take a new orthonormal basis $\{u_{1} = e_{1},u_{2},u_{3}\}$ such that $u_{2}$ and $u_{3}$ are eigenvectors of $\ad_{e_{1}}$. Hence, the bracket operation is expressed as
$[u_{1},u_{2}] = \alpha u_{2}$, $[u_{1},u_{3}] = \beta u_{3}$, $ [u_{2},u_{3}] = 0$, where $\alpha$ and $\beta$ are the corresponding eigenvalues. Then $G$ must be isomorphic to the orthogonal semidirect product $\R\ltimes_{\pi}\R^{2}$ such that $\pi(t) = \E^{t\,\diag(\alpha, \beta)} = \diag\big( \E^{\alpha\, t}, \E^{\,\beta t}\big)$.

According to Example \ref{solvable-gen-1}, the group $ \R\ltimes_{\pi}\R^{2}$ admits a description as the matrix group
 $G^3(\alpha,\beta)$, and the  left-invariant metric making $\{u_{1},u_{2},u_{3}\}$ an orthonormal basis is that given in the statement of Theorem \ref{classi} (2) below.  

For $\alpha = \beta$, $G^3(\alpha,\beta)$  is the three-dimensional hyperbolic space with constant sectional curvature $-\alpha^{2}$, and for the unimodular case $\alpha = -\beta$, putting $u_{1}' = u_{1}$, $u_{2}' = (1/\sqrt{2})(u_{2}+u_{3})$, $u_{3}' = (1/\sqrt{2})(u_{2}-u_{3})$, one gets
\[
[u_{1}',u_{2}'] = \alpha u_{3}',\quad [u_{2}',u_{3}'] = 0,\quad [u_{3}',u_{1}'] = -\alpha u_{2}'.
\]
Hence, $G^3(\alpha,-\alpha)$  is isometrically isomorphic to $E(1,1)$ equipped with a left-invariant Riemannian metric with
principal Ricci curvatures $(0,0,-2\alpha^2)$  (see \eqref{Ricci-Galphai}). Then, we have the following result.

\begin{theorem}
\label{classi}
A three-dimensional connected, simply-connected  nonabelian cyclic metric Lie group   is isometrically isomorphic to one of the following Lie groups with a suitable left-invariant metric:

\noindent $(1)$ The group $\widetilde{SL(2,\R)}$ with the family of metrics depending on two positive parameters $\lambda_{1},\lambda_{2}$ making orthonormal the bases of Ricci eigenvectors described in \emph{(\ref{basis})} and with principal Ricci curvatures $(-2\lambda_2(\lambda_1+\lambda_2),-2\lambda_1(\lambda_1+\lambda_2)$, $2\lambda_1\lambda_2)$.

\smallskip

\noindent $(2)$ The  orthogonal semidirect product $\R\ltimes_{\pi}\R^{2}$, both factors with the additive group structure and where the action $\pi$ is 
$\pi(t) = \diag\big( \E^{\alpha\, t}, \E^{\,\beta t}\big)$, $(\alpha,\beta)\in \R^{2}\setminus \{(0,0)\}$. 
This Lie group can be described as the matrix group  $G^3(\alpha,\beta)$
with the left-invariant metric 
\[
g = \D u^{2} + \E^{-2\alpha u}\D x^{2} + \E^{-2\beta u}\D y^2,
\]
where $(u,x,y)$ corresponds to the global coordinate system 
$(u,x_1,x_2)$ in \emph{Example \ref{solvable-gen-1}}.

The case $\alpha = \beta$ corresponds to the hyperbolic space $H^{3}(\alpha)$ and the case $\alpha = -\beta$ to the group $E(1,1)$ of rigid motions of the Minkowski plane equipped with a one-parameter family of left-invariant metrics with signature of the Ricci form $(0,0,-)$. Moreover,  $G^3(\alpha,0)$  $($resp.\  $G^3(0,\beta)$$)$ is isometrically isomorphic to the direct product of the $2$-dimensional real hyperbolic space $H^{2}(\alpha)$ $($resp.\ $H^{2}(\beta)$$)$ and $\R$.
\end{theorem}

Note that for dimension three, according to \cite[Corollary  2.2]{KowTri}, a nonsymmetric  manifold admitting a nontrivial
structure of class $\T_3$ also admits a nontrivial structure of class $\T_2$.

The classification of connected, simply-connected cyclic metric Lie groups for dimension four was given by Kowalski and Tricerri \cite[Theorem  3.1]{KowTri} for the unimodular case. We add the corresponding nonunimodular Lie groups in the following theorem. 

\begin{theorem}
\label{four-nonunim}
A four-dimensional connected and simply-connected nonabelian cyclic metric Lie group $G$ is isometrically
isomorphic to one of the following metric Lie groups: \smallskip

\noindent $(1)$ The direct product $\widetilde{SL(2,\R)}\times \R$, where 
$\widetilde{SL(2,\R)}$ is equipped with any of the cyclic left-invariant metrics described in \emph{Theorem \ref{classi} (1)}. \smallskip

\noindent $(2)$ The matrix Lie  group  $G^4(\alpha,\beta,\gamma)$, $(\alpha,\beta,\gamma)\in\R^3\setminus\{(0,0,0)\}$, 
with the left-inva\-riant metric 
\[
g = \D u^{2} + \E^{-2\alpha u}\D x^{2} + \E^{-2\beta u}\D y^2 + \E^{-2\gamma u}\D z^2,
\]
where $(u,x,y,z)$ corresponds to the global coordinate system 
$(u,x_1,x_2,x_3)$ in \emph{Example \ref{solvable-gen-1}}. 

The Lie group  $G^4(\alpha,-\alpha,0)$  corresponds to the metric direct product 
$E(1,1)\times\R$, where $E(1,1)$ is  equipped with a left-invariant Riemannian metric
 with principal Ricci curvatures $(0,0,-2\alpha^2)$.

\smallskip

\noindent $(3)$ The orthogonal semidirect product   $\R^2\ltimes_{\pi}\R^{2}$ under $\pi\colon\R^2\rightarrow \Aut(\R^2)$ given by 
$\pi(s,t) = \diag(\E^{\,\rho s+\lambda t}, \E^{\sigma s-\lambda t})$, $\rho+\sigma\not=0$, $\lambda>0$. 
This Lie group can be described as the group $H^4(\rho,\sigma;\lambda)$ of matrices of the form 
\[
\left(
\begin{array}{ccc}
\E^{\,\rho u+\lambda v} & 0 & x \\
0 & \E^{\sigma u-\lambda v} & y \\
\noalign{\smallskip}
0 & 0 & 1
\end{array}
\right),  \qquad  \rho+\sigma\not=0, \quad \lambda>0,
  \]
  with the left-invariant metric
\[
g = \D u^{2} + \D v^{2}+ \E^{-2(\rho u+\lambda v)}\D x^{2} +
\E^{-2(\sigma u-\lambda v)}\D y^2. 
\]
\end{theorem}

\begin{proof}
From \cite[Proposition 3.6]{KowTri}, one has that a four-dimensional connected,
simply-connected,  unimodular, nonabelian  cyclic metric Lie group $G$ is  either
$\widetilde{SL(2,\R)}\times\R$ or $E(1,1)\times\R$ or 
 $G^4(\alpha,\beta,\gamma)$  for $(\alpha,\beta,\gamma)\in \R^{3}\setminus \{(0,0,0)\}$
with $\alpha+\beta+\gamma=0$.

Suppose now that the Lie algebra $\fg$ of $G$ is not unimodular, then its unimodular kernel~$\fu$ has dimension three and $\fg$   is  the semidirect sum  $\fu^\bot\oplus \fu$ under the adjoint representation,  orthogonal with respect to the inner product  $\langle\cdot,\cdot\rangle$ defining the left-invariant Riemannian metric $g$ on $G$. The unimodular ideal
 $\fu$ of $\fg$, with its induced inner product, must be isometrically isomorphic to either
 $\fsl(2,\R)$ or $\fe(1,1)$ or $\R^3$, and there exists an orthonormal basis  $\{e_{0},e_{1},e_{2},e_{3}\}$
 of $\fg$ such that $\fu^\bot=\R\{e_0\}$ and $\fu=\R\{e_1,e_2,e_3\}$ such that
 \begin{equation}
 \label{uni}
[e_{2},e_{3}]=\lambda_{1}e_{1},\quad [e_{3},e_{1}] = \lambda_{2}e_{2},\quad [e_{1},e_{2}] = \lambda_{3}e_{3},
\qquad \lambda_1+\lambda_2+\lambda_3=0.
\end{equation}
Since $g$ is cyclic left-invariant, $\ad_{e_0}$ is  selfadjoint on $\fu$, and then
\begin{equation} 
\label{ade-4}
\ad_{e_0} = \left(
\begin{array}{ccc}
\mu_1 & a_3 & a_2\\
a_3 & \mu_2 & a_1 \\
a_2 & a_1 & \mu_3
\end{array}
\right ),
\end{equation}
in terms of the basis $\{e_{1},e_{2},e_{3}\}$ of $\fu$. Since $\ad_{e_0}$ acts as a derivation on the Lie algebra
$\fu$, we get
\begin{equation} \label{lambdas}
  \lambda_{\,k}(\mu_i+\mu_j-\mu_k) = 0,   \quad (\lambda_{\,i}+\lambda_j)a_k=0,
\end{equation}
for each cyclic permutation $(i,j,k)$ of $(1,2,3)$.

{\it First case}. If $\fu=\fsl(2,\R)$ then $\lambda_1\lambda_2\lambda_3\not=0$,
and for all distinct $i,j,k$, we have $\lambda_{\,k}=-(\lambda_{\,i}+\lambda_j)$,
then $a_k=0$ and $\mu_i=\mu_j=\mu_k=0$. Hence $\ad_{e_0}=0$ and $\fg$
would be the direct sum $\R\oplus\fsl(2,\R)$, which is unimodular.

{\it Second case}. If $\fu$ is the abelian Lie algebra $\R^3$,
we can take an orthonormal basis of eigenvectors $\{u_1,u_2,u_3\}$
of the  selfadjoint operator $\ad_{e_0}$ acting on~$\fu$.
The structure equations are given, in terms of the orthonormal
basis $\{e_0,u_1,u_2,u_3\}$ of~$\fg$, by
\[
[e_{0},u_{1}] = \alpha u_{1},\quad
[e_{0},u_{2}] = \beta u_{2},\quad
[e_{0},u_{3}] = \gamma u_{3},
\]
with $\alpha+\beta+\gamma\not=0$, because $\fg$ is not unimodular.
 This gives (see Example~\ref{solvable-gen-1}) the nonunimodular version of (2)  in the statement.

{\it Third case}. If $\fu=\fe(1,1)$ we can suppose $\lambda_1=-\lambda_2>0$, $\lambda_3>0$ in~\eqref{uni}.
By~\eqref{lambdas}, we have $a_1=a_2=0$, $\mu_3=0$ and $\mu_1=\mu_2$ in~\eqref{ade-4},
and put $\lambda=\lambda_1=-\lambda_2$, $\mu=\mu_1=\mu_2$, $a=a_3$.
Then $[e_0,e_1]=\mu e_1+a e_2$, $[e_0,e_2]=ae_1+\mu e_2$, $[e_2,e_3]=\lambda e_1$, 
  $[e_1,e_3]=\lambda e_2$,  
and $\R\{e_0,e_3\}$ and $\R\{e_1,e_2\}$ are mutually orthogonal abelian subalgebras 
of~$\fg$. We now consider the orthonormal basis $\{u_1,u_2,u_3,u_4\}$ of $\fg$ given by $u_{1} = e_0$, 
$u_{2} =-e_3$, $u_{3} = (1/\sqrt 2)(e_{1}+e_{2})$, $u_{4} = (1/\sqrt{2})(e_{1}-e_{2})$. 
Then
\begin{equation} \label{brackets-ui-4}
  [u_1,u_3]=\rho u_3, \quad [u_1,u_4]=\sigma u_4, \quad
  [u_2,u_3]=\lambda u_3, \quad [u_2,u_4]=-\lambda u_4,
\end{equation}
where we have put $\rho=\mu+a$, $\sigma=\mu-a$, so that, since $\fg$ is not
unimodular, $\rho+\sigma=2\mu\not=0$. Thus, $\fg$ is the direct sum of its
abelian subalgebra $\R\{u_1,u_2\}$ and its abelian ideal $\R\{u_3,u_4\}$.
Hence,  the Lie group $G$ is isomorphic to the orthogonal semidirect product
$\R^2  \ltimes_\pi \R^2$, where the action $\pi$ is 
$\pi(s,t) = \diag(\E^{\,\rho s+\lambda t}, \E^{\sigma s-\lambda t})$,  
that is, $G$ can be considered as $\R^4$ with the group operation
\[
  (u,v,x,y)\cdot(u',v',x',y')=(u+u',\;v+v',\;x+\E^{\,\rho u+\lambda v}x',\;y+\E^{\sigma u-\lambda v}y'),
\]
and it admits a  description  as the matrix group  $H^4(\rho,\sigma;\lambda)$.
With respect to the global coordinate system $(u,v,x,y)$ of  $H^4(\rho,\sigma;\lambda)$,
the generators $u_i$ of $\fg$ correspond to the left-invariant vector fields 
$u_{1} = \partial/\partial u$, $u_{2} = \partial/\partial v$, 
$u_{3} = \E^{\,\rho u+\lambda v} \partial/\partial x$, $u_{4} = \E^{\sigma u-\lambda v}\partial/\partial y$.  
\end{proof}

\begin{remark} \rm One can prove that, if $\rho\sigma=\lambda^2$,
then  $H^4(\rho,\sigma;\lambda)$ is isometrically isomorphic to the metric direct product Lie group 
$H^2(\alpha)\times H^2(\beta)$, with $\alpha= \sqrt{\lambda^2+\rho^2}$ and
$\beta= \lambda(\rho+\sigma)/\sqrt{\lambda^2+\rho^2}$ and that $H^4(\rho,-\rho;\lambda)$ 
is isometrically isomorphic to the direct product $E(1,1)\times \R$, where $E(1,1)$ is equipped 
with a left-invariant Riemannian metric with principal Ricci curvatures $\big(0,0,-2(\lambda^{2} + \rho^{2})\big)$.  

In some cases, the  nonunimodular  Lie group  $G^4(\alpha,\beta,\gamma)$  is
a decomposable  metric Lie group. In fact,  $G^4(\alpha,\beta,0)$  is isometrically
isomorphic to the direct product  $G^3(\alpha,\beta)\times\R$, and
 $G^4(\alpha,\alpha,0)$  and  $G^4(\alpha,0,0)$  are isometrically isomorphic to the direct
products $H^3(\alpha)\times\R$ and $H^2(\alpha)\times\R^2$, respectively.
On the other hand, the indecomposable metric Lie group
 $G^4(\alpha,\beta,\gamma)$  for $\alpha=\beta=\gamma$
is the four-dimensional hyperbolic space $H^4(\alpha)$.
\end{remark}

The classification of unimodular Lie algebras of dimension five corresponding 
to traceless cyclic metric Lie groups was given by Bieszk \cite{Bie}.
We add the corresponding nonunimodular Lie groups in the next theorem. 

\begin{theorem}
\label{theofive}
A five-dimensional connected and simply-connected nonabelian cyclic metric Lie group $G$  is isometrically
isomorphic to one of the following metric Lie groups: \smallskip

\noindent $(1)$ The metric direct products $\widetilde{SL(2,\R)} \times \R^2$ and 
$\widetilde{SL(2,\R)}\times H^2(\alpha)$, where $\widetilde{SL(2,\R)}$ is equipped with any 
of the cyclic left-invariant metrics described in \emph{Theorem 6.1 (1)}.\smallskip

\noindent $(2)$ The Lie group  $G^5(\alpha_1,\alpha_{\,2},\alpha_3,\alpha_{\,4})$, for
     $(\alpha_1,\alpha_{\,2},\alpha_3,\alpha_{\,4})\in  \R^{4}$, with some $\alpha_{\,i}\not=0$. \smallskip

\noindent $(3)$ The orthogonal semidirect product   $\R^2\ltimes_{\pi}\R^{3}$ under $\pi\colon\R^2\rightarrow\Aut(\R^3)$
given by $\pi(s,t) = \diag( \E^{\,\rho s +\lambda t}, \E^{\,\sigma s + \mu t}, \E^{\,\tau s - (\lambda+\mu)t})$. 
This Lie group can be described as the group  $H^5(\rho,\sigma,\tau;\lambda,\mu)$ of matrices of the form 
\begin{equation}
\label{rstlm}
\left(
\begin{array}{cccc}
\E^{\,\rho u+\lambda v} &  0 & 0 & x \\
\noalign{\smallskip}
0 & \E^{\,\sigma u+\mu v} & 0 & y \\
\noalign{\smallskip}
0 & 0  & \E^{\,\tau u - (\lambda + \mu) v} & z \\
0 & 0  & 0              & 1
\end{array}
\right ),
\end{equation} 
with the left-invariant metric
\[ 
g = \D u^{2} + \D v^{2}+ \E^{-2(\rho u +\lambda v)}\D x^{2}+\E^{-2(\sigma u+\mu v)}\D y^{2} + \E^{2((\lambda+\mu) v-\tau u)}\D z^2 .
\]
In particular, if $\rho+\sigma\not=0$ and $\lambda\not=0$,  $H^5(\rho,\sigma,0;\lambda,-\lambda)$  is
isometrically isomorphic to the direct product   $H^4(\rho,\sigma;\lambda)\times\R$.
\end{theorem}

\begin{proof}
Let $g$ be the cyclic left-invariant metric on $G$, and let $\langle\cdot,\cdot\rangle$ be
the corresponding inner product  on the
 Lie algebra $\fg$ of $G$.
First, suppose that $\fg$ is not unimodular and let $\fu$ be its unimodular kernel. Then $\fg$ is the semidirect
sum $\fu^\bot \oplus \fu$ under the adjoint representation, orthogonal with respect to $\langle\cdot,\cdot\rangle$. With its induced inner product, $\fu$
 must be isometrically isomorphic to either the direct sum $\fsl(2,\R)\oplus\R$ or $\R^4$
 or the Lie algebra $\fg(\alpha,\beta,\gamma)$, with $\alpha+\beta+\gamma=0$.
 In each case, there exists a suitable orthonormal  basis  $\{e_{0},e_{1},e_{2},e_{3},e_{4}\}$
 of $\fg$ such that $\fu^\bot=\R\{e_0\}$ and $\fu=\R\{e_1,e_2,e_3,e_4\}$.
 Since $g$ is cyclic left-invariant, $\ad_{e_0}$ is  selfadjoint on $\fu$, and
 in terms of the basis $\{e_{1},e_{2},e_{3},e_{4}\}$ of $\fu$ we can write
\begin{equation}  \label{matriz}
\ad_{e_0} = \left(
\begin{array}{cccc}
\mu_1 & a_3 & a_2 & a_4\\
a_3 & \mu_2 & a_1 & a_5\\
a_2 & a_1 & \mu_3 & a_6\\
a_4 & a_5 & a_6   & \mu_4
\end{array}
\right ).
\end{equation}

We have to consider three cases. 

 \emph{First case}. If $\fu=\fsl(2,\R)\oplus\R$ we can suppose that
  \begin{equation*}
[e_{2},e_{3}]=\lambda_{1}e_{1},\quad [e_{3},e_{1}] = \lambda_{2}e_{2},\quad [e_{1},e_{2}] = \lambda_{3}e_{3},
\quad [e_i,e_4]=0, \qquad i=1,2,3,
\end{equation*}
with all $\lambda_{\,i}$ not zero and
$\lambda_1+\lambda_2+\lambda_3=0$.
Since $\ad_{e_0}$ acts as a derivation on
$\fu$, we get
$\lambda_{\,k}(\mu_i+\mu_j-\mu_k) = 0$, $(\lambda_{\,i}+\lambda_j)a_k=0$,
$\lambda_1 a_4= \lambda_2 a_5= \lambda_3 a_6= 0$, 
for each cyclic permutation $(i,j,k)$ of $(1,2,3)$. Then $\mu_1=\mu_2=\mu_3=0$ and
$a_i=0$ in~\eqref{matriz} for each $i=1,\ldots,6$. If $\mu_4=0$ in~\eqref{matriz} then $\fg$
should be be the unimodular Lie algebra $\fsl(2,\R)\oplus\R^2$.
 Then $\mu_4\not=0$ and $\fg$ must be the orthogonal direct sum of the Lie algebra
$\fsl(2,\R)$ and the Lie algebra $\R\{e_0,e_4\}$ with the structure equation
$[e_0,e_4]=\mu_4 e_4$, and hence we have $(1)$ in the statement,
with $\alpha=\mu_4$.

\emph{Second case}. If $\fu=\R^4$, we can take a basis formed by eigenvectors
$u_1, u_2,u_3$, $u_4$ of the selfadjoint operator
$\ad_{e_0}$ acting on $\fu$. In terms of the orthonormal basis
$\{e_0,u_1,u_2$, $u_3,u_4\}$ of~$\fg$, the structure equations of $\fg$ are given by 
\[
[e_0,u_i]=\alpha_{\,i} u_i, \qquad i=1,\dotsc,4, 
\]
with $\alpha_1+\alpha_{\,2}+\alpha_3+\alpha_{\,4}\not=0$  because $\fg$ is not unimodular.
In accordance with Example~\ref{solvable-gen-1}, we have the metric Lie groups in $(2)$ in the statement with the condition
$\sum_{i=1}^4\alpha_{\,i}\not=0$.

\emph{Third case}. If $\fu=\fg(\alpha,\beta,\gamma)$, with $(\alpha,\beta,\gamma)\not=(0,0,0)$, $\alpha+\beta+\gamma=0$,
we can suppose that
\[
  [e_4,e_1]=\alpha e_1,\quad  [e_4,e_2]=\beta e_2,\quad [e_4,e_3]=\gamma e_3.
\]
Then, according to the discussion in Example~\ref{solvable-gen-2}, the 
nonunimodular Lie group $G$ can be described as the group of matrices of the
form \eqref{rstlm}, 
 with $\rho+\sigma+\tau\not=0$ and $(\lambda,\mu)\not=(0,0)$. 

Finally, from the classification of unimodular Lie algebras of dimension five with a cyclic 
left-invariant metric, given by Bieszk \cite{Bie}, we  also have the metric direct product 
Lie group $\widetilde{\SL(2,\R)} \times \R^2$ in (1) in the statement, and the metric Lie group  
$G^5(\alpha_1,\alpha_{\,2},\alpha_3,\alpha_{\,4})$  in~$(2)$ with
 $\sum_{i=1}^4\alpha_{\,i}=0$ (which includes the cases~(c) and~(e) in~\cite[Theorem  1]{Bie}).
  The  metric Lie group  $H^5(\rho,\sigma,\tau;\lambda,\mu)$ in~$(3)$ with
 $\rho+\sigma+\tau=0$  corresponds to the case~(d)  in~\cite[Theorem~1]{Bie};
in fact, it can be described (as a particular case of the description~\eqref{H-sigmai-mui} 
of the matrix Lie group in Example~\ref{solvable-gen-2}) as the group
  $\hat{H}^5(\sigma,-\sigma,\mu_1,\mu_2,-\mu_1-\mu_2)$  of matrices of the form
 \[
  \left(
\begin{array}{cccc}
\E^{\mu_1 v} & 0 & 0 & x\\
0 & \E^{\sigma u+\mu_2 v} & 0 & y \\
0 & 0 & \E^{-\sigma u-(\mu_1+\mu_2)v} & z\\
0 &  0  &  0&  1
\end{array}
\right ),
\]
with $\mu_1,\sigma\not=0$, which gives the Lie group in case~(d)  in~\cite[Theorem  1]{Bie}.
\end{proof}

\begin{remark}
Many decomposable nonunimodular metric Lie groups appear as particular cases of the metric Lie groups
 $H^5(\rho,\sigma,\tau;\lambda,\mu)$.
 For example, if $\rho+\sigma\not=0$, $\tau=0$, $\lambda\not=0$, and $\mu=-\lambda$, we have the direct  product  $H^4(\rho,\sigma;\lambda)\times\R$.

Moreover, one can prove that $H^2(c)\times G^3(\alpha,\beta)$  is isometrically isomorphic to $H^5(\rho,\sigma,\tau;\lambda,\mu)$.
In particular, if $\alpha=\beta$ this metric Lie group corresponds to $H^2(c)\times H^3(\alpha)$.
If $\beta=0$ we have  $H^5(\rho,\sigma,0;\lambda,-\lambda)$, which is isometrically isomorphic to $H^2(c)\times H^2(\alpha)\times\R$.
If $\alpha=-\beta$, we have the metric Lie group  $H^5(c,0,0;0,\!-\alpha)$, which corresponds to
the direct product $H^2(c)\times E(1,1)$,
where $E(1,1)$ is  equipped with  a left-invariant Riemannian metric
 with principal Ricci curvatures
$(0,0,-2\alpha^2)$.
\end{remark}

\smallskip

\begin{remark} 
As for dimensions greater than or equal to six, we recall some facts.

Let $(M,g)$ be an $n$-dimensional connected Riemannian manifold $(M,g)$ admitting a structure $S \in \T_1 \oplus \T_2$.
Defining the vector field $\xi$ on $M$ by $\xi = (1/(n-1))\sum_{i=1}^n S_{e_i}e_i$, for any local orthonormal basis $\{e_i\}$,
then $S$ can be written as $S_XY = g(X,Y)\xi - g(Y,\xi)X + \pi(X,Y)$, where $\pi$ is a tensor field of type $(1,2)$.
The one-form $\omega$ metrically dual to the vector field $\xi$ is called the fundamental form of $S$.

Pastore and Verroca \cite{PasVer} studied
proper structures $S$ of type $\T_1 \oplus \T_2$ (that is, $S$ belongs neither to $\T_1$ nor to $\T_2$),
having closed fundamental $1$-form $\omega$. They proved (among other results) that a connected, simply-connected Riemannian manifold
which is a warped product and admits a nontrivial such structure $S$, is isometric to the real hyperbolic
space $H^n$ of constant curvature $-{\Vert\xi\Vert}^2$ and $n\geq 6$. However, the description of $H^n$ as a Riemannian homogeneous space corresponding to the proper structure $S\in \T_1 \oplus \T_2$ does not correspond to that of a cyclic metric Lie group, since the corresponding structure on the latter is necessarily vectorial. 
\end{remark}

{\footnotesize

\noindent {\bf Authors's addresses:}

\smallskip

\noindent (P.M.G.) Instituto de F\'\i sica Fundamental, CSIC, Serrano 113-bis, 28006-Ma\-drid, Spain. {\it E-mail:} pmgadea@iff.csic.es         
          
\smallskip

\noindent (J.C.G-D.) Department of Fundamental Mathematics, University of La Laguna, 38200-La Laguna, Tenerife, Spain. {\it E-mail:} jcgonza@ull.es
          
\smallskip

\noindent (J.A.O.) Departamento de Xeometr\'{\i}a e Topolox\'\i a, Facultade de Matem\'a\-ticas, Universidade de Santiago de Compos\-tela,
    15782-Santiago de Compostela, Spain. {\it E-mail:} ja.oubina@usc.es

}

\end{document}